\documentclass[5p,authoryear]{article}

\usepackage{hyperref}
\hypersetup{
	colorlinks=true,
	linkcolor=blue,
	filecolor=blue,      
	urlcolor=black,
	citecolor=blue,
	anchorcolor=black,
	menucolor=red
}
\usepackage{dsfont}
\usepackage{amsmath}
\usepackage{amsthm}
\usepackage{amssymb}
\usepackage{mathrsfs}
\usepackage[english]{babel}
\usepackage{comment}
\usepackage{enumerate}
\usepackage{geometry}
\usepackage{pythonhighlight}
\usepackage{graphicx}
\usepackage{fancyhdr}
\usepackage{url}
\usepackage{float}
\usepackage{booktabs}
\usepackage[all]{xy}
\usepackage{color}
\usepackage{natbib}
\usepackage{hyperref}
\usepackage{authblk}
\newcommand{\ud}{\mathrm{d}}

\newcommand{\X}{\mathcal {X}}

\newcommand{\E}{\mathbb{E}}

\newcommand{\R}{\mathbb{R}}

\newcommand{\p}{\mathbb{P}}
\newcommand{\Q}{\mathbb{Q}}

\newcommand{\M}{\mathcal{M}}

\renewcommand{\geq}{\geqslant}
\renewcommand{\leq}{\leqslant}
\renewcommand{\epsilon}{\varepsilon}

\newcommand\keywords[1]{\textbf{Keywords}: #1}
\theoremstyle{plain}
\newtheorem{theorem}{Theorem}

\newtheorem{lemma}{Lemma}

\theoremstyle{definition}

\newtheorem{assumption}{Assumption}

\theoremstyle{remark}
\newtheorem{remark}{Remark}
\newtheorem{problem}{Problem}

\begin{document}
\title{ Monotone Mean-Variance Portfolio Selection in  Semimartingale Markets: Martingale Method }

\author[a]{Yuchen Li}
\author[a]{Zongxia Liang}
\author[b]{Shunzhi Pang}
\affil[a]{\small{Department of Mathematical Sciences, Tsinghua University, Beijing 100084, China}}
\affil[b]{\small{School of Economics and Management, Tsinghua University, Beijing 100084, China}}
  \maketitle

		\begin{abstract}

   We use the martingale method to discuss the relationship between mean-variance (MV) and monotone mean-variance (MMV) portfolio selections. We propose a unified framework to discuss the relationship in general financial markets without any specific setting or completeness requirement. We apply this framework to a semimartingale market and find that MV and MMV are consistent if and only if the variance-optimal signed martingale measure keeps non-negative. Further, we provide an example to show the application of our result. 
		\end{abstract}
		

%


	\section{Introduction}
	
	Since \citet{markowits1952portfolio}  earliest establishes the mean-variance (MV) portfolio selection criterion, it has been widely applied in both theoretical research and practical investment. However, due to the well-known drawback of non-monotonicity (see  \citet{dybvig1982mean}), portfolio selection based on MV preferences may violate the economic rationality and even lead to the free cash flow stream (FCFS) problem. To conquer such disadvantage,  \citet{maccheroni2009portfolio} propose the monotone mean-variance (MMV) preference as a revision, which is the best approximation of MV satisfying the monotonicity property.
	
	In recent years, portfolio selection based on MMV preferences is becoming widely concerned. On the one hand, much literature focuses on the relationship between MV and MMV. By the dynamic programming method, \citet{trybula2019continuous} first solve the continuous-time MMV portfolio selection in a market with a stochastic factor and find that solutions to MV and MMV problems are actually the same. Then, \citet{SL2020note} and \citet{du2023monotone} take an analytical approach to generalize such consistency of MV and MMV to any market with continuous asset prices, regardless of whether there are trading constraints. Thus, discontinuity should be a necessary condition for MV and MMV to be different. By adding L\'evy jumps into \citet{trybula2019continuous}'s model, \citet{li2022comparison} solve the MMV investment problem in a discontinuous market for the first time and propose a sufficient and necessary condition for MV and MMV to be inconsistent. Finally, in general semimartingale financial markets, \citet{cerny2020semimartingale} constructs a truncated quadratic utility function to connect MV and MMV and provides several equivalent conditions for them to be consistent. 
	
	On the other hand, the literature is also devoted to obtaining an explicit form of solutions to MMV problems, as such result can be applied into practice directly. For instance, \citet{li2024MMVPS} consider a standard jump-diffusion model and solve the MMV portfolio selection explicitly. They find that MMV can fix the non-monotonicity and FCFS problems of MV when the jump size can be larger than the inverse of the market price of risk. So far, however, almost all the literature on solving continuous-time MMV problems follows the dynamic programming method proposed by \citet{trybula2019continuous} (also see \citet{LG2021RAIRO-OR}, \citet{li2023optimal}, \citet{shen2022cone}, \citet{hu2023constrained}). However, the martingale method, as another classical way widely applied to deal with stochastic control problems, has been neglected. 
	
	In this paper, we fill such theoretical gap and use the perspective of martingale method to discuss the relationship between MV and MMV for the first time. First, we show the application of the martingale method and prove that MV and MMV are consistent in any complete financial market. Then, we propose a unified framework to discuss the relationship in general financial markets without any specific setting or requirement of completeness. For the MV investment problem, it has been widely discussed in the literature, see \citet{XYincomplete2006} for instance. For the MMV problem, we assume that the investor calculates its utility by choosing a particular pricing kernel, which gives an upper bound of the value function of MMV problem. If there exists a proper pricing kernel such that the upper bound coincides with the value function of MV problem, then MV and MMV problems share the same solution. If such pricing kernel does not exist, the value functions of MV and MMV problems would have a gap greater than 0, and therefore the two solutions are different. We apply this framework to a semimartingale market. We give out the solution to MMV problem and conclude that MV and MMV are consistent if and only if the variance-optimal signed martingale measure (VSMM) keeps non-negative. Finally, we give an example to show the validity and practicability of our method. 
	
	To summarize, we contribute to the literature by providing a new perspective to deal with MMV portfolio selection problems. While almost all the literature on MMV adopt the dynamic programming method proposed by \citet{trybula2019continuous}, we apply the martingale method for the first time. Such method is not only mathematically simpler, but also makes more sense economically. We believe that it can be universally applied to most of other MV and MMV problems with different market settings. \citet{cerny2020semimartingale} is the closest to this paper, as we come to almost the same conclusion. It takes the method of supremal convolution to transform the MMV problem into a truncated quadratic utility investment problem. Such complicated mathematical technique partly covers up the economic nature of MMV. Instead, using the martingale method, we can clearly observe that taking a signed measure to discount the investor's wealth process is the fundamental reason for the non-monotonicity of MV, consistent with \citet{li2024MMVPS}. Our method is also more general and can be used in different financial markets. 
	
	The remainder of the paper is organized as follows. Section \ref{sec com} takes the martingale method to prove the consistency of MV and MMV in any complete market. Section \ref{sec incom} establishes a unified framework in general markets based on the same perspective and apply the framework to a semimartingale market. A specific example is also given out to show the application of our result. Finally, Section \ref{sec conclusion} concludes the paper.
	
	\section{Complete Market}\label{sec com}
	
	In this section, we first show the validity of the martingale method in a complete financial market, which leads to a better understanding of further discussions on general markets. 
	
	Suppose that the financial market is arbitrage free. Based on the fundamental theorem of asset pricing, there exists as least one risk neutral probability measure that is equivalent to the physical measure. When the market is complete, the risk neutral measure is unique and a strictly positive stochastic discounted factor (which is also known as the pricing kernel) can be constructed. By the classical martingale method of asset pricing, the investor's final stage wealth $X^{\pi}_T$ under an admissible self-financing strategy $\pi$ should satisfy 
	\begin{equation}
		x_0 = \mathbb{E} \left[ M X^{\pi}_T\right], \label{Wealth in Complete Market}
	\end{equation}
	where $x_0$ is the investor's initial endowment and $M$ denotes the pricing kernel at time $T$. Conversely, a claim $X^{\pi}_T$ satisfying Equation \eqref{Wealth in Complete Market} can be replicated by an admissible strategy and we denote as $\pi\in\mathcal{A}$. Generally, we keep the following assumption in this section:
	\begin{assumption}\label{assump complete}
		There exists a unique pricing kernel $M>0$ with $\E\left[M^2\right]<\infty$. 
	\end{assumption}
	
	\subsection{MV Problem in Complete Market}\label{subsec com MV}
	
	First, we derive the solution to the MV investment problem in the complete market: 
	\begin{problem}[MV problem in complete market]\label{prob com MV 1}
		\begin{displaymath}
			\left\{\begin{aligned}
				&\max_{X^{\pi}_T}\  U_\theta(X^{\pi}_T),  \\
				&\text{subject to }\mathbb{E} \left[ M X^{\pi}_T\right]=x_0, 
			\end{aligned}\right.
		\end{displaymath}	
		where $ U_\theta(X^{\pi}_T)=\E\left[X^{\pi}_T\right]-\frac{\theta}{2}\mathrm{Var}(X^{\pi}_T)$, and $\theta$ is the uncertainty averse coefficient. 
	\end{problem}
	
	Although the problem is widely studied in the literature, we restate the result here for the purpose of completeness. We embed the problem into auxiliary problems to simplify the problem (see, e.g., Chapter 6 in \citet{ZXY1999stochastic}, or \citet{trybula2019continuous}, \citet{li2022comparison}). We add an extra condition $\E\left[X^{\pi}_T\right]=d$, then use the Lagrange multiplier method: 
	\begin{displaymath}
		\max_{X^{\pi}_T} \Bigg\{d - \frac{\theta}{2} \mathbb{E} \Big( X^{\pi}_T - d \Big)^2  + \lambda_1 \left( x_0 - \mathbb{E} \left[ M X^{\pi}_T \right] \right) + \lambda_2 \Big(\mathbb{E} \left[ X^{\pi}_T\right] - d \Big) \Bigg\},
	\end{displaymath}
	where $\lambda_1$ and $\lambda_2$ are corresponding Lagrangian multipliers. 
	
	Solving the first order condition with respect to $X^{\pi}_T$, we have
	$X^{\pi^*}_T(d) = d + \frac{1}{\theta} \left( \lambda_2 - \lambda_1 M \right).$
	Taking the expectation, we have $\lambda_2 = \lambda_1 \mathbb{E} \left[ M \right]$.
	As 
	\begin{displaymath}
		\begin{aligned}
			x_0 =& \mathbb{E} \left[ M X^{\pi^*}_T(d) \right]  \\
			=& \mathbb{E} \left\{ M \left[d + \frac{\lambda_1}{\theta}\left(\mathbb{E} \left[ M \right] - M\right) \right] \right\} \\
			= &d \mathbb{E} \left[ M \right] - \frac{\lambda_1}{\theta} \mathrm{Var}\left[ M \right],
		\end{aligned}
	\end{displaymath}
	
	we have $\lambda_1 = \frac{\theta \left( d \mathbb{E} \left[ M \right] - x_0 \right) }{\mathrm{Var}\left[ M \right]}$
	and thus 
	\begin{displaymath}
		X^{\pi^*}_T(d) = d + \frac{\left( d \mathbb{E} \left[ M \right] - x_0 \right) \left( \mathbb{E} \left[ M \right] - M  \right) }{\mathrm{Var}\left[ M \right]}.
	\end{displaymath}
	For simplicity, we denote $\mathbb{E} \left[ M \right] = m_1$, $\mathrm{Var}\left[ M \right] = m_2$. 
	
	For the original Problem \ref{prob com MV 1}, we just solve the optimization over $d$:
	\begin{displaymath}
		\max_{d} U_\theta(X^{\pi^*}_T(d))=\max_{d} \Bigg\{ d - \frac{\theta}{2} \frac{\left(d m_1 - x_0\right)^2  }{m_2 }\Bigg\}.
	\end{displaymath}
	Simple computation yields $d^{\ast} = \frac{x_0}{m_1} + \frac{m_2}{\theta m_1^2 }$,  then
	\begin{equation}\label{MV Optimal Wealth in Complete} X^{\pi^{\ast}}_T=X^{\pi^{\ast}}_T(d^*) = \frac{x_0}{m_1} + \frac{m_2}{\theta m_1^2 } + \frac{m_1 -  M}{\theta m_1 }.
	\end{equation}
	
	As the domain of monotonicity of MV preference is  $\mathcal{G}_{\theta}=\left\{X\in L^2:X-\E\left[X\right]\leq \frac{1}{\theta}\right\}$ (see \citet{maccheroni2009portfolio} for more details), it can be verified that the optimally invested wealth given by Equation \eqref{MV Optimal Wealth in Complete} keeps within the monotone domain, i.e.,
	\begin{displaymath}
		X^{\pi^{\ast}}_T- \mathbb{E} \left[X^{\pi^{\ast}}_T \right] = \frac{m_1 - M }{\theta m_1 } < \frac{1}{\theta},
	\end{displaymath}
	due to the strict positivity of $M$. 
	
	Finally, the maximum MV utility gained by the investor is
	\begin{equation} \label{eq MV value function}
		U_\theta\left(X^{\pi^{\ast}}_T \right) = \frac{x_0}{m_1} + \frac{m_2}{2 \theta m_1^2 }.
	\end{equation}
	
	\subsection{MMV Problem in Complete Market} \label{Subsec com MMV}
	
	Recall that the MMV preference (see \citet{maccheroni2009portfolio} for more details) is given by
	\begin{displaymath}
		V_{\theta}(X)=\inf_{\mathbb{Q} \in \Delta^2(\mathbb{P})} \left\{ \mathbb{E^{Q}}[X]+\frac{1}{2\theta}\mathbb{C}(\mathbb{Q} \parallel \mathbb{P}) \right\}, \quad X \in L^2(\Omega,\mathcal{F},\mathbb{P}), \label{MMV Definition}
	\end{displaymath}
	where $\mathbb{Q} \in \Delta^2(\mathbb{P}) \triangleq \left\{\mathbb{Q}:\mathbb{Q}(\Omega)=1, \mathbb{E}^\mathbb{P} \left[{\left(\frac{\mathrm{d}\mathbb{Q}}{\mathrm{d}\mathbb{P}}\right)}^2\right] < \infty \right\}$ and
	\begin{displaymath}
		\mathbb{C}(\mathbb{Q} \parallel \mathbb{P}) \triangleq \left\{
		\begin{aligned}
			& \mathbb{E^\mathbb{P}}\left[{\left(\frac{\mathrm{d}\mathbb{Q}}{\mathrm{d}\mathbb{P}}\right)}^2 \right]-1, & \quad \text{if} \quad \mathbb{Q} \ll \mathbb{P}, \\
			& \infty, & \quad \text{otherwise}.
		\end{aligned} \right. \notag
	\end{displaymath}
	By the martingale method, the MMV investment problem in the complete market is: 
	\begin{problem}[MMV problem in complete market] \label{prob com MMV 1}
		\begin{displaymath}
			\left\{\begin{aligned}
				&\max_{X^{\pi}_T }\  V_\theta(X^{\pi}_T)=\max_{X^{\pi}_T }\inf_{\mathbb{Q} \in \Delta^2(\mathbb{P})}\mathbb{E}^{\mathbb{Q}} \left\{ X^{\pi}_T + \frac{1}{2\theta} \frac{\mathrm{d}\mathbb{Q}}{\mathrm{d}\mathbb{P}} - \frac{1}{2\theta}\right\} ,\\
				&\text{subject to }\mathbb{E} \left[ M X^{\pi}_T\right]=x_0.
			\end{aligned}\right. 
		\end{displaymath}
	\end{problem}
	
	We have the following theorem of the relationship between MV and MMV investment problems in any complete market: 
	\begin{theorem}
		In any complete market (with Assumption \ref{assump complete}), solutions to MV Problem \ref{prob com MV 1} and MMV Problem \ref{prob com MMV 1} are the same.
	\end{theorem}
	\begin{proof}
		Given $\Q\in\Delta^2(\p)$, we consider an auxiliary preference $\bar{V}_{\theta,\Q}$ by choosing the probability measure $\Q$ in $V_\theta$:
		\begin{displaymath}
			\bar{V}_{\theta,\Q}(X) =  \mathbb{E}^{\mathbb{Q}} \left\{ X + \frac{1}{2\theta} \frac{\mathrm{d}\mathbb{Q}}{\mathrm{d}\mathbb{P}} - \frac{1}{2\theta}\right\} , \quad X \in L^2(\Omega,\mathcal{F},\mathbb{P}),
		\end{displaymath}
		where $\Q$ is given for the investor to take an expectation of its wealth with a penalty $\frac{\mathrm{d}\mathbb{Q}}{\mathrm{d}\mathbb{P}} - \frac{1}{2\theta}$. Obviously, we have $V_\theta(X)\leq \bar{V}_{\theta,\Q}(X) $ for any $\Q$ and $X$.
		
		Now we consider a particular measure $Q^{\ast}$ given by
		\begin{displaymath}
			\mathrm{d}\mathbb{Q}^{\ast}(\omega) = \frac{M}{m_1} \mathrm{d}\mathbb{P}(\omega), \quad \omega \in \Omega,
		\end{displaymath}
		which is actually the risk neutral measure corresponding to the pricing kernel. Interestingly, under $\Q^*$, the auxiliary preference $\bar{V}_{\theta,\Q^*}(X)$  is actually a constant, as
		\begin{equation}\label{eq MMV constant}
			\begin{aligned}
				&\mathbb{E}^{\mathbb{Q}^{\ast}} \left[ X^{\pi}_T + \frac{1}{2\theta} \frac{\mathrm{d}\mathbb{Q}^{\ast}}{\mathrm{d}\mathbb{P}} - \frac{1}{2\theta} \right] \\
				&= \mathbb{E} \left[ X^{\pi}_T \frac{\mathrm{d}\mathbb{Q}^{\ast}}{\mathrm{d}\mathbb{P}} + \frac{1}{2\theta} \left(\frac{\mathrm{d}\mathbb{Q}^{\ast}}{\mathrm{d}\mathbb{P}}\right)^2 - \frac{1}{2\theta} \right]\\ &= \frac{x_0}{m_1} + \frac{m_2}{2 \theta m_1^2 }.
			\end{aligned}
		\end{equation}
		Based on the property of the MMV preference, for any claim $X^\pi_T$ satisfying the admissible constraint $\mathbb{E} \left[ M X^{\pi}_T\right]=x_0$, we have the inequality
		\begin{displaymath}
			U_{\theta}\left(X^{\pi}_T\right) \leq V_{\theta}\left(X^{\pi}_T\right) \leq \bar{V}_{\theta,\Q^*}\left(X^{\pi}_T\right).
		\end{displaymath}
		Then,
		\begin{displaymath}
			\max_{\pi\in\mathcal{A}}U_{\theta}\left(X^{\pi}_T\right) \leq \max_{\pi\in\mathcal{A}}V_{\theta}\left(X^{\pi}_T\right) \leq \max_{\pi\in\mathcal{A}}\bar{V}_{\theta,\Q^*}\left(X^{\pi}_T\right).
		\end{displaymath}
		However, using Equation \eqref{eq MV value function} and  Equation \eqref{eq MMV constant}, we have
		\begin{displaymath}
			U_{\theta}\left(X^{\pi^{\ast}}_T\right) = \max_{\pi \in\mathcal{A}} V_{\theta}\left(X^{\pi}_T\right) =  \bar{V}_{\theta,\Q^*}\left(X^{\pi^{\ast}}_T\right) = \frac{x_0}{m_1} + \frac{m_2}{2 \theta m_1^2 }, 
		\end{displaymath}
		which indicates that the optimal values of Problem \ref{prob com MV 1} and Problem \ref{prob com MMV 1} are the same. 
	\end{proof}
	
	
	\section{General Market}\label{sec incom}
	
	When the market is more general, the discounted process (pricing kernel) may not be unique and the optimally invested wealth may not be replicated. Therefore, the above discussion needs to be revised. In this section, we first propose a unified framework for general markets on which we give a theorem on the relationship of MV and MMV. Second, we consider a semimartingale market and use the framework to give an equivalent condition of the consistency of MV and MMV. Finally, we give a nontrivial example to verify our result.
	
	\subsection{General Framework}\label{subsec general}
	
	Consider the following two general optimization problems. Let $\mathcal{A}$ denote the appropriate admissible set for the investor's final wealth $X^\pi_T$.  We only assume that $\mathcal{A}$ is a closed subset of $L^2(\p)$ and do not require any specific market setting. 
	\begin{problem}[General MV problem]\label{prob gen MV}
		\begin{displaymath}
			\left\{\begin{aligned}
				&\max_{}\  U_\theta(X_T),  \\
				&\text{subject to }X_T\in \mathcal{A}.
			\end{aligned}\right.
		\end{displaymath}
	\end{problem}
	\begin{problem}[General MMV problem]\label{prob gen MMV}
		\begin{displaymath}
			\left\{\begin{aligned}
				&\max_{}\  V_\theta(X_T),  \\
				&\text{subject to }X_T\in \mathcal{A}.
			\end{aligned}\right.
		\end{displaymath}
	\end{problem}
	
	Suppose that the two problems are finite. Then, there exists an optimally invested wealth $X^{*}_T\in\mathcal{A}$ that reaches the optimal value of MV Problem \ref{prob gen MV} because of the closed property of $\mathcal{A}$. Denote $\M(x_0)$ as the set of (signed) discounting measures under which the expected value of all the admissible $X_T$ is a constant $x_0\in\R$:
	\begin{displaymath}
		\M(x_0)=\left\{g\in L^2(\p):\E\left[g\right]=1, \E\left[gX_T\right]=x_0,\ \forall X_T\in \mathcal{A} \right\}.
	\end{displaymath}
	Let $L^2_+(\p)=\left\{g\in L^2(\p):g\geq 0 \right\}$. Then, we have the following characterization on the relationship of MV and MMV:
	\begin{theorem}\label{thm general}
		If the optimally invested wealth $X^{*}_T$ of MV Problem \ref{prob gen MV} satisfies 
		\begin{displaymath}
			1-\theta \left(X^{*}_T-\E\left[X^{*}_T\right]\right)\in L^2_+(\p)\cap \M(x_0) \text{ for some }x_0\in R,
		\end{displaymath} 
		then the optimal values of MV Problem \ref{prob gen MV} and MMV Problem \ref{prob gen MMV} are the same, i.e., 
		\begin{displaymath}
			\max_{X_T\in \mathcal{A}}\  U_\theta(X_T)=\max_{X_T\in \mathcal{A}}\  V_\theta(X_T).
		\end{displaymath}
		If 
		\begin{displaymath}
			1-\theta \left(X^{*}_T-\E\left[X^{*}_T\right]\right)\notin L^2_+(\p),
		\end{displaymath}
		then the optimal values of MV Problem \ref{prob gen MV} and MMV Problem \ref{prob gen MMV} are different, i.e., 
		\begin{displaymath}
			\max_{X_T\in \mathcal{A}}\  U_\theta(X_T)<\max_{X_T\in \mathcal{A}}\  V_\theta(X_T).
		\end{displaymath}
	\end{theorem}
	\begin{proof}
		If $1-\theta \left(X^{*}_T-\E\left[X^{*}_T\right]\right)\in L^2_+(\p)\cap \M(x_0)$, similar to Subsection \ref{Subsec com MMV}, we consider a particular measure in $V_\theta$. Denote $M=1-\theta \left(X^{*}_T-\E\left[X^{*}_T\right]\right)$. Under the discounting measure $\frac{\ud \Q^*}{\ud \p}=M\in \M(x_0)$, the utility under $\bar{V}_\theta$ again becomes a constant as in Equation \eqref{eq MMV constant}:
		\begin{displaymath}
			\begin{aligned}
				&\mathbb{E}^{\mathbb{Q}^{\ast}} \left[ X_T + \frac{1}{2\theta} \frac{\mathrm{d}\mathbb{Q}^{\ast}}{\mathrm{d}\mathbb{P}} - \frac{1}{2\theta} \right] \\
				&= \mathbb{E} \left[ X^{\pi}_T \frac{\mathrm{d}\mathbb{Q}^{\ast}}{\mathrm{d}\mathbb{P}} + \frac{1}{2\theta} \left(\frac{\mathrm{d}\mathbb{Q}^{\ast}}{\mathrm{d}\mathbb{P}}\right)^2 - \frac{1}{2\theta} \right] \\&= {x_0} + \frac{\E\left[M^2\right]}{2 \theta}- \frac{1}{2\theta}, 
			\end{aligned}
		\end{displaymath}
		which gives an upper bound of the utility under MMV. However, simple calculation yields 
		\begin{align}
			U_\theta(X^{*}_T)&=\inf_{Y\in L^2,\E\left[Y\right]=1}\E\left[X^{*}_TY+\frac{1}{2\theta}Y^2-\frac{1}{2\theta}\right]\notag \\
			&=\E\left[X^{*}_TM+\frac{1}{2\theta}M^2-\frac{1}{2\theta}\right]={x_0} + \frac{\E\left[M^2\right]}{2 \theta}- \frac{1}{2\theta}. \notag
		\end{align} 
		Thus, the following inequality holds:
		\begin{displaymath}
			\max_{X_T\in\mathcal{A}}U_\theta(X_T)\leq \max_{X_T\in\mathcal{A}}V_\theta(X_T)\leq U_\theta(X^{*}_T)=\max_{X_T\in\mathcal{A}}U_\theta(X_T).
		\end{displaymath}  
		As such, 
		\begin{displaymath}
			\max_{X_T\in\mathcal{A}}V_\theta(X_T)=\max_{X_T\in\mathcal{A}}U_\theta(X_T),
		\end{displaymath}
		thus we get the consistency of MV and MMV. 
		
		Conversely, if $1-\theta \left(X^{*}_T-\E\left[X^{*}_T\right]\right)\notin L^2_+(\p)$, then
		\begin{displaymath}
			\p(X^*_T-\E\left[X^*_T\right]> \frac{1}{\theta})>0,
		\end{displaymath} 
		which means that $X^*_T$ is not in the domain of monotonicity of MV preference. Then, we have $U_\theta(X^*_T)<V_\theta(X^*_T)$ and 
		\begin{displaymath}
			\max_{X_T\in\mathcal{A}}U_\theta(X_T)=U_\theta(X^*_T)< \max_{X_T\in\mathcal{A}}V_\theta(X_T),
		\end{displaymath}
		which indicates the inconsistency of MV and MMV. 	
	\end{proof}
	
	As we can see, Theorem \ref{thm general} does not rely on any specific setting on the financial market, the risky asset, or the admissible investment strategy $\pi$, but only gives a characterization on deciding whether MV and MMV problems are consistent or not.

	\subsection{Semimartingale Market}
	In this subsection, we consider a semimartingale market and use our general framework to give an equivalent characterization of the relationship between MV and MMV. We consider a continuous-time multi-dimensional semimartingale $S$, which represents multiple risky assets in the financial market. The setting is the same as \citet{XYincomplete2006}, where the classical MV portfolio selection in an incomplete market has been studied. 
	
	Let $K^{S}_2$ be the subspace spanned by the ``simple'' strategies $h\cdot(S_{T_2}-S_{T_1})$, where $T_1\leq T_2$ are stopping times such that $\{S_U: U \text{ is a stopping time}, U\leq T_2\}\subset L^2(\p)$ and $h\in L^\infty(\p)$. Let $K_2$ be the closure of $K_2^S$ in $L_2$. It has been proved that there exists a set $\mathcal{M}^{2,s}$ which represents all signed martingale measures:
	\begin{displaymath}
		\mathcal{M}^{2,s}=\left\{g\in L^2(\p):\E\left[gf\right]=0,\text{ for all }f\in K^{S}_2\text{ and } \E\left[g\right]=1\right\}.
	\end{displaymath}
	
	Denoted $\Theta$ as all  $\mathbb{R}^{m}$-valued predictable $S$-integrable processes $\vartheta$, such that $G_{T}(\vartheta):=\int_{0}^{T} \vartheta d S \in L^{2}(\mathbb{P})$  and  $\mathbb{E}\left[G_{T}(\vartheta) g\right]=0, \ \forall \ g \in \mathcal{M}^{2,s}$. Let
	\begin{displaymath}
		G_{T}(\Theta):=\left\{G_{T}(\vartheta): \vartheta \in \Theta\right\}.
	\end{displaymath}
	Under the following two assumptions, we have $G_{T}(\Theta) = K_{2}$ and the dual relationship 
	\begin{equation}\label{eq incom dual relation}
		f\in K_2\Longleftrightarrow f\in L^2(\p), \quad \E\left[fg\right]=0\text{ for all }g\in \mathcal{M}^{2, s}.  
	\end{equation}
	
	Throughout the rest part of this section, we keep the following two assumptions:
	\begin{assumption}
		$S$ is locally in $L^2(\p)$.
	\end{assumption}
	\begin{assumption}
		The set of all equivalent martingale measures $\mathcal{M}^{2,e}=\left\{q\in \mathcal{M}^{2,s}:q>0\right\}$ is non-empty.
	\end{assumption}
	
	Then, $\mathcal{M}^{2, s}$ is a non-empty, closed and convex subset of $L^{2}(\mathbb{P})$. There exists a unique element of the minimal $L^{2}(\mathbb{P})$ norm in $\mathcal{M}^{2, s}$, denoted by $M$ and called the variance-optimal signed martingale measure (VSMM) for $S$, as it minimizes $\operatorname{Var}[g]$ over $g \in \mathcal{M}^{2, s}$. 
	
	A trading strategy $\vartheta$ is said to be admissible if $\vartheta=\left(\vartheta^1_t,\dots,\vartheta^n_t\right)_{0\leq t\leq T}$ is a predictable $S$-integral process and $\vartheta\in \Theta$. Suppose that the investor has an initial wealth $x_0$, then the discounted wealth process can be represented as
	\begin{displaymath}
		X_t^{x_0,\vartheta}=x_0+\int_{0}^{t} \vartheta d S,\ 0\leq t\leq T.
	\end{displaymath} 
	We denote all admissible wealth processes by $\X(x_0)=\{X^{x_0,\vartheta}:\vartheta\in\Theta\}$. The set of all the admissible final wealth $\left\{X_T:X\in \X(x_0)\right\}$ is correspondingly $x_0+G_{T}(\Theta)$. Up to here, all definitions required in our general framework (Subsection \ref{subsec general}) are clear. When the initial value is $x_0$, the admissible final wealth set is $\mathcal{A}=x_0+G_{T}(\Theta)$. And the set of discounting measures is $\M(x_0)=\M^{2,s}$.
	
	\subsection{MV Problem in Semimartingale Market} \label{subsec incom MV}
	
	To apply Theorem \ref{thm general}, we need the optimally invested wealth under MV. The MV investment problem in semimartingale market has been partly solved by \citet{XYincomplete2006} without giving out the optimal strategy or value function. In this subsection, we restate and further refine their results. We use the notation $X\in \X(x_0)$ instead of $X_T\in \mathcal{A}=x_0+G_{T}(\Theta)$ to match the notation in \citet{XYincomplete2006}. The MV investment problem is: 
	\begin{problem}[MV problem in semimartingale market]\label{prob incom MV 1}\begin{displaymath}
			\left\{\begin{aligned}
				&\max\  U_\theta(X_T)=\E\left[X_T\right]-\frac{\theta}{2}\mathrm{Var}(X_T),\\
				&\text{subject to } X \in \mathcal{X}(x_0).
			\end{aligned}\right.
		\end{displaymath}
	\end{problem}
	
	We embed this problem to a series of auxiliary problems: 
	\begin{problem}\label{prob incom MV 2}
		\begin{displaymath}\left\{\begin{aligned}
				&\min\  U_\theta(X_T)=\E[X_T^2],\\
				&\text{subject to } X \in \mathcal{X}(x_0), \E\left[X_T\right]=z. 
			\end{aligned}\right.\end{displaymath} 
	\end{problem}
	\begin{problem} \label{prob MV 3}
		\begin{displaymath}\left\{\begin{aligned}
				&\min\  U_\theta(X_T)=\E[(X_T-\lambda)^2],\\
				&\text{subject to } X \in \mathcal{X}(x_0).
			\end{aligned}\right.\end{displaymath} 
	\end{problem}
	
	\citet{XYincomplete2006} have proved that the optimally invested wealth of Problem \ref{prob incom MV 2} has a one-to-one mapping relation to that of Problem \ref{prob MV 3}. Then, they give the following lemma on the optimally invested wealth of Problem \ref{prob MV 3}:
	\begin{lemma}\label{lem MV prob3}
		Let $\lambda\neq x_0$. If there exists an $\tilde{X}\in\X(x_0)$, $\tilde{g}\in\M^{2,s}$ such that 
		\begin{displaymath}
			\tilde{X}_T=\lambda-\tilde{y}\tilde{g},
		\end{displaymath} 
		then $\tilde{X}$ solves Problem \ref{prob MV 3} and $\tilde{y}=\frac{\lambda-x_0}{\E\left[M^2\right]}$. 
	\end{lemma}
	
	Based on Lemma \ref{lem MV prob3}, we propose the following theorem, which is slightly different from the result of \citet{XYincomplete2006}:
	\begin{theorem}
		The optimally invested wealth that solves Problem \ref{prob MV 3} is
		\begin{displaymath}
			X^{\lambda,*}_T=\lambda-\frac{\lambda-x_0}{\E\left[M^2\right]}M.
		\end{displaymath}
		Meanwhile, there exists an $X\in\X(x_0)$ such that $X_T=X^{\lambda,*}_T$.
	\end{theorem}
	
	\begin{proof}
		By Lemma \ref{lem MV prob3}, we only need to prove that there exists an $X\in\X(x_0)$ such that $X_T=X^{\lambda,*}_T$. For every $g\in\M^{2,s}$, we have   
		\begin{displaymath}
			\begin{aligned}
				&\E\left[\left(X^{\lambda,*}_T-x_0\right)g\right]\\&=\lambda-x_0-\frac{\lambda-x_0}{\E\left[M^2\right]}\E\left[Mg\right]\\
				&=\lambda-x_0-\frac{\lambda-x_0}{\E\left[M^2\right]}\E\left[M^2\right]\\&=0. 
			\end{aligned}
		\end{displaymath}
		Thus, the dual relationship Equation \eqref{eq incom dual relation} indicates that $\exists \ \vartheta^*$ such that $X^{\lambda,*}_T-x_0=\int_{0}^{T}\vartheta^*\ud S$.
	\end{proof}
	
	After solving Problem \ref{prob MV 3}, we are able to derive the optimally invested wealth of Problem \ref{prob incom MV 1} and the following theorem gives the result:
	\begin{theorem}\label{thm mv semi}
		The optimally invested wealth that solves Problem \ref{prob incom MV 1} is 
		\begin{displaymath}
			X^*_T=x_0+\frac{\E\left[M^2\right]}{\theta}-\frac{1}{\theta}M. 
		\end{displaymath}
		The corresponding value function is 
		\begin{displaymath} 
			U_\theta(X^{*}_T)=x_0+\frac{\E\left[M^2\right]}{2\theta}-\frac{1}{2\theta}.
		\end{displaymath}
	\end{theorem}
	
	\begin{proof}
		It is easy to see that the optimally invested terminal wealth that solves Problem \ref{prob incom MV 1} should belong to the following set:
		\begin{displaymath}
			\left\{X^{\lambda,*}_T=\lambda-\frac{\lambda-x_0}{\E\left[M^2\right]}M:\lambda\in\R\right\}. 
		\end{displaymath}
		The MV utility gained on $X^{\lambda,*}_T$ is 
		\begin{displaymath} 
			U_\theta(X^{\lambda,*}_T)=\lambda-\frac{\lambda-x_0}{\E\left[M^2\right]}-\frac{\theta}{2}\frac{(\lambda-x_0)^2}{\E\left[M^2\right]^2}\left(\E\left[M^2\right]-1\right), 
		\end{displaymath}
		which is a quadratic function with the parameter $\lambda$. And it is easy to check that the $\lambda$ that maximizes $U_\theta(X^{\lambda,*}_T)$ is $\lambda^*=x_0+\frac{\E\left[M^2\right]}{\theta}$, and then 
		\begin{displaymath} 
			X^*_T=x_0+\frac{\E\left[M^2\right]}{\theta}-\frac{1}{\theta}M, \quad U_\theta(X^{*}_T)=x_0+\frac{\E\left[M^2\right]}{2\theta}-\frac{1}{2\theta}. 
		\end{displaymath}
	\end{proof}
	
	\subsection{MMV Problem in Semimartingale Market}
	
	In this subsection, we consider the solution to the MMV investment problem and then discuss the relationship of MV and MMV in the semimartingale market. We use the result in our general framework to give an equivalent characterization of this relationship. Following the setting in Subsection \ref{subsec incom MV}, the problem is constructed as: 
	\begin{problem}[MMV problem in semimartingale market]\label{prob MMV}
		\begin{displaymath} 
			\left\{\begin{aligned}
				&\max\  {V}_{\theta}(X_T) = \inf_{\mathbb{Q} \in \Delta^2(\mathbb{P})}\mathbb{E}^{\mathbb{Q}} \left\{ X_T + \frac{1}{2\theta} \frac{\mathrm{d}\mathbb{Q}}{\mathrm{d}\mathbb{P}} - \frac{1}{2\theta}\right\},\\
				&\text{subject to }X\in \mathcal{X}(x_0).
			\end{aligned}\right.
		\end{displaymath} 
	\end{problem}

	Now we can directly use Theorem \ref{thm general} to get the relationship between MV and MMV in this semimartingale market:
	\begin{theorem}
		Solutions to MV Problem \ref{prob incom MV 1} and MMV Problem \ref{prob MMV} are the same if and only if the variance-optimal signed martingale measure (VSMM) keeps non-negative, i.e., $M \geq 0$.
	\end{theorem}
	\begin{proof}
		By Theorem \ref{thm mv semi} we already know that the optimally invested wealth under MV Problem \ref{prob incom MV 1} is $X^*_T=x_0+\frac{\E\left[M^2\right]}{\theta}-\frac{1}{\theta}M$. Then, we have \begin{displaymath}
			1-\theta\left(X^*_T-\E\left[X^*_T\right]\right)=M\in \M=\M^{2,s}.
		\end{displaymath}
		Thus, by Theorem \ref{thm general}, if $M\in L^2_+(\p)$, solutions to MV Problem \ref{prob incom MV 1} and MMV Problem \ref{prob MMV} are the same. If $M\notin L^2_+(\p)$, solutions to two problems are different.
	\end{proof}
	
	In this market, the particular martingale measure $\frac{\ud \Q^*}{\ud \p}$ used by us is $M$. Under this measure, the investor discounts the wealth process and decides the investment strategy. If the measure is signed, then the MV investor would take a signed measure to discount its wealth process, which is not reasonable in the economic rationality and results in the non-monotonicity problem. That is the fundamental reason for the non-monotonicity of MV preferences, leading to the inconsistency of MV and MMV. Instead, MMV fixes such problem by taking an adjusted and non-signed measure to discount the investor's wealth process, see the following Theorem \ref{thm semi mmv}.   
	
	Furthermore, even though MV and MMV can be inconsistent, we can still give out the solution to MMV problem. 
	\begin{theorem}\label{thm semi mmv}
		For Problem \ref{prob MMV}, the maximum utility reached by the MMV investor is 
		\begin{displaymath}
			\max_{X\in\X(x_0)}V_\theta(X_T)=x_0+\frac{1}{2\theta}\E\left[\tilde{M}^2\right]-\frac{1}{2\theta},
		\end{displaymath}
		where $\tilde{M}=\arg\min_{Y\in \M^{2,s}\cap  L^2_+}\E\left[Y^2\right]$. The optimally invested wealth should satisfy the dual relationship 
		\begin{displaymath}
			\tilde{M}=\theta\left(X_T^*-x_0-\frac{\E\left[\tilde{M}^2\right]}{\theta}\right)^-.
		\end{displaymath}
		Equivalently, we have 
		\begin{displaymath}
			X_T^*\wedge \kappa-\E\left[X_T^*\wedge \kappa\right]=\frac{1}{\theta}-\frac{\tilde{M}}{\theta},
		\end{displaymath} 
		where  $\E\left[\left(X_T^*-\kappa\right)^-\right]=\frac{1}{\theta}$  and  $\kappa=x_0+\frac{\E\left[\tilde{M}^2\right]}{\theta}$ .    
	\end{theorem}
	\begin{proof}
		For any $\Q$ such that $\frac{\ud \Q}{\ud\p}\in \M^{2,s}\cap  L^2_+$, we have 
		\begin{displaymath}
			V_\theta(X_T)\leq x_0+\frac{1}{2\theta}\E\left[\left(\frac{\ud \Q}{\ud\p}\right)^2\right]-\frac{1}{2\theta}.
		\end{displaymath} 
		Then,  
		\begin{displaymath}
			\begin{aligned}
				\max_{X\in\X(x_0)}V_\theta(X_T) &\leq x_0+\frac{1}{2\theta}\min_{Y\in \M^{2,s}\cap  L^2_+}\E\left[Y^2\right]-\frac{1}{2\theta} \\&=x_0+\frac{1}{2\theta}\E\left[\tilde{M}^2\right]-\frac{1}{2\theta}.
			\end{aligned}
		\end{displaymath}
		
		If we can find an admissible $X^*\in\X(x_0)$ such that $V_\theta(X_T^*)=x_0+\frac{1}{2\theta}\E\left[\tilde{M}^2\right]-\frac{1}{2\theta}$, then the inequality holds. It is equivalent to find an admissible $X^*\in\X(x_0)$ such that 
		\begin{displaymath}
			X_T^*\wedge \kappa-\E\left[X_T^*\wedge \kappa\right]=\frac{1}{\theta}-\frac{\tilde{M}}{\theta},
		\end{displaymath} 
		where 
		\begin{displaymath}
			\E\left[\left(X_T^*-\kappa\right)^-\right]=\frac{1}{\theta}.
		\end{displaymath}
		
		Let $\kappa=x_0+\frac{\E\left[\tilde{M}^2\right]}{\theta}$. Consider a truncated utility function 
		\begin{equation}
			\tilde{U}(x)=\left\{\begin{aligned}
				&x-\frac{1}{2\kappa}x^2,\ & x \leq \kappa, \\
				&\frac{\kappa}{2},\ & \text{else}, 
			\end{aligned}\right.\notag
		\end{equation}
		or equivalently $\tilde{U}(x)=x\wedge\kappa-\frac{1}{2\kappa}(x\wedge\kappa)^2$. Using Theorem 4.10 in \citet{biagini2011admissible} on $\tilde{U}$ and $S$, we can get a claim that satisfies all the requirements.
	\end{proof}
	
	\begin{remark}
		While $M$ is the variance optimal signed martingale measure, $\Tilde{M}$ is the variance optimal absolute continuous martingale measure of this semimartingale market. That is the key difference between them, which also corresponds to the essential difference between MV and MMV. Under MMV, the investor takes the non-negative $\Tilde{M}$ as the pricing kernel and therefore avoids the non-monotonicity problem. The (in)consistency of $M$ and $\Tilde{M}$ depends on the market setting. When the market is complete or asset prices are continuous, $M$ and $\Tilde{M}$ are actually the same. Otherwise, they could be different, see the example given in Subsection \ref{subsec example}.
	\end{remark}
	
	\subsection{Example} \label{subsec example}
	
	In this subsection, we give an example to show the validity and practicability of our method.
	
	We consider an incomplete market with discontinuous sample paths used in \citet{li2024MMVPS}, which is described by the following system:
	\begin{displaymath}
		\left\{
		\begin{aligned}
			& {\mathrm{d}S_0(t)} ={S_0(t)} r \mathrm{d} t, \\
			& \mathrm{d}S_1(t) ={S_1(t^-)}\Bigg \{ \mu \mathrm{d}t + \sigma \mathrm{d} B(t) + \mathrm{d} \sum_{i = 1}^{N(t)}Q_i\Bigg  \},
		\end{aligned}
		\right.\end{displaymath}
	and the dynamic of the wealth process corresponding to a self-financed strategy $\pi$ is 
	\begin{displaymath}
		\begin{aligned}
			\mathrm{d} X^{\pi} (t) 
			 = &\Big[X^{\pi} (t)r + \pi(t)(\mu-r + \lambda \xi_1) \Big]\mathrm{d}t + \pi(t)\sigma\mathrm{d}B(t) \\&+ \pi(t^-) \int_{-1}^{\infty} q \widetilde{N}(\mathrm{d}t, \mathrm{d}q),
		\end{aligned}
	\end{displaymath}
	where $L(t)=\sum_{i = 1}^{N(t)}Q_i$ is a compound Poisson process and $\tilde{N}$ is the corresponding compensated random measure. See \citet{li2024MMVPS} for more details about the market model.
	
	Taking the dynamic programming method to get the HJBI equation of the value function, \citet{li2024MMVPS} solve out the explicit optimal strategy and value function of the MMV problem. They compare the optimal strategies and value functions of MV and MMV problems and then give an equivalency condition on the consistency of two problems. However, based on the general framework proposed by this paper, we only need the solution of the MV problem to give the same equivalency condition. 
	
	Referring to \citet{li2024MMVPS}, the optimally invested wealth under MV is
	\begin{displaymath}
		X^{mv,\ast}(T)  = e^{Tr}x_0 + \frac{e^{TC}}{\gamma} 
		- \frac{1}{\gamma}  M(T),
	\end{displaymath} 
	where 
	\begin{displaymath}
		\begin{aligned}
			M(s)=\mathcal{E} \Bigg( &-\int_0^. \frac{\left(\mu - r + \lambda \xi_1\right)\sigma}{\sigma^2 + \lambda \xi_2^2} \mathrm{d}B(\tau) \\&- \int_0^.\int_{-1}^{\infty} \frac{\left(\mu - r + \lambda \xi_1\right)q}{\sigma^2 + \lambda \xi_2^2} \widetilde{N}(\mathrm{d}\tau,\mathrm{d}q) \Bigg)_s
		\end{aligned}
	\end{displaymath}
	is a Dol\'eans-Dade exponential satisfying 
	\begin{displaymath}
		\begin{aligned}
			\ud M(s)&=M(s^-)\Bigg( - \frac{\left(\mu - r + \lambda \xi_1\right)\sigma}{\sigma^2 + \lambda \xi_2^2} \mathrm{d}B(\tau) \\&- \int_{-1}^{\infty} \frac{\left(\mu - r + \lambda \xi_1\right)q}{\sigma^2 + \lambda \xi_2^2} \widetilde{N}(\mathrm{d}\tau,\mathrm{d}q) \Bigg),\\ M(0)&=1.
		\end{aligned}
	\end{displaymath}
	
	Simple calculation yields that $\E\left[X^\pi(T)M(T)\right]=e^{Tr}x_0$ for all the admissible $X^\pi(T)$ and then $M(T)$ is a discounting measure in $\M(e^{Tr}x_0)$ defined in Subsection \ref{subsec general}. Applying Theorem \ref{thm general}, we know that MV and MMV are consistent if and only if $M(T)\geq 0$, which is equivalent to $\Delta L(t)\leq \frac{\sigma^2 + \lambda \xi_2^2}{\mu - r + \lambda \xi_1}=\bar{q}$. Such result is the same as Theorem 7 in \citet{li2024MMVPS}.
	
	Moreover, as \citet{li2024MMVPS} obtain the explicit solution of MMV problem, it can be verified that Theorem \ref{thm semi mmv} holds true for the financial market with jumps satisfying $\Delta L(t)> \frac{\sigma^2 + \lambda \xi_2^2}{\mu - r + \lambda \xi_1}=\bar{q}$, which results in the inconsistency of MV and MMV. Therefore, this example shows the application of the martingale method and verifies the validity of our result.

	\section{Conclusion} \label{sec conclusion}
	
	In this paper, we provide a new perspective based on the martingale method to discuss the relationship between MV and MMV portfolio selection problems. We first prove the consistency of two problems in any complete market. Then we propose a unified framework to discuss the relationship in general financial markets without any special market setting or completeness requirement. We apply this framework to a semimartingale market and find that solutions to MV and MMV problems coincide if and only if the VSMM keeps non-negative. We also derive the solution to MMV problem when it differs with MV problem. We verify that taking a signed measure to discount the investor's wealth process is the fundamental reason for the non-monotonicity of MV. And MMV fixes such problem by taking an adjusted and non-signed measure to discount the wealth process. Finally, we provide a nontrivial example to show the validity and practicability of our method. We believe that the martingale method and the analysis framework we propose in this paper can be universally applied to most of other MV and MMV problems with different market settings.

 \section{Acknowledgments}
The authors acknowledge the support from the National Natural Science Foundation of China (Grant No.12271290, and No.11871036). The authors also thank the members of the group of Actuarial Sciences and Mathematical Finance at the Department of Mathematical Sciences, Tsinghua University for their feedback and useful conversations.

	\bibliographystyle{elsarticle-num-names}
	\bibliography{ORL-MMV}

\begin{thebibliography}{16}
\expandafter\ifx\csname natexlab\endcsname\relax\def\natexlab#1{#1}\fi
\providecommand{\url}[1]{\texttt{#1}}
\providecommand{\href}[2]{#2}
\providecommand{\path}[1]{#1}
\providecommand{\DOIprefix}{doi:}
\providecommand{\ArXivprefix}{arXiv:}
\providecommand{\URLprefix}{URL: }
\providecommand{\Pubmedprefix}{pmid:}
\providecommand{\doi}[1]{\href{http://dx.doi.org/#1}{\path{#1}}}
\providecommand{\Pubmed}[1]{\href{pmid:#1}{\path{#1}}}
\providecommand{\bibinfo}[2]{#2}
\ifx\xfnm\relax \def\xfnm[#1]{\unskip,\space#1}\fi
\bibitem[{Markowits(1952)}]{markowits1952portfolio}
\bibinfo{author}{H.~M. Markowits},
\newblock \bibinfo{title}{Portfolio selection},
\newblock \bibinfo{journal}{Journal of Finance} \bibinfo{volume}{7}
  (\bibinfo{year}{1952}) \bibinfo{pages}{71--91}.
\bibitem[{Dybvig and Ingersoll(1982)}]{dybvig1982mean}
\bibinfo{author}{P.~H. Dybvig}, \bibinfo{author}{J.~E. Ingersoll},
\newblock \bibinfo{title}{Mean-variance theory in complete markets},
\newblock \bibinfo{journal}{Journal of Business}  (\bibinfo{year}{1982})
  \bibinfo{pages}{233--251}.
\bibitem[{Maccheroni et~al.(2009)Maccheroni, Marinacci, Rustichini, and
  Taboga}]{maccheroni2009portfolio}
\bibinfo{author}{F.~Maccheroni}, \bibinfo{author}{M.~Marinacci},
  \bibinfo{author}{A.~Rustichini}, \bibinfo{author}{M.~Taboga},
\newblock \bibinfo{title}{Portfolio selection with monotone mean-variance
  preferences},
\newblock \bibinfo{journal}{Mathematical Finance} \bibinfo{volume}{19}
  (\bibinfo{year}{2009}) \bibinfo{pages}{487--521}.
\bibitem[{Trybu{\l}a and Zawisza(2019)}]{trybula2019continuous}
\bibinfo{author}{J.~Trybu{\l}a}, \bibinfo{author}{D.~Zawisza},
\newblock \bibinfo{title}{Continuous-time portfolio choice under monotone
  mean-variance preferences—stochastic factor case},
\newblock \bibinfo{journal}{Mathematics of Operations Research}
  \bibinfo{volume}{44} (\bibinfo{year}{2019}) \bibinfo{pages}{966--987}.
\bibitem[{Strub and Li(2020)}]{SL2020note}
\bibinfo{author}{M.~S. Strub}, \bibinfo{author}{D.~Li},
\newblock \bibinfo{title}{A note on monotone mean--variance preferences for
  continuous processes},
\newblock \bibinfo{journal}{Operations Research Letters} \bibinfo{volume}{48}
  (\bibinfo{year}{2020}) \bibinfo{pages}{397--400}.
\bibitem[{Du and Strub(2023)}]{du2023monotone}
\bibinfo{author}{J.~Du}, \bibinfo{author}{M.~S. Strub},
\newblock \bibinfo{title}{Monotone and classical mean-variance preferences
  coincide when asset prices are continuous},
\newblock \bibinfo{journal}{Available at SSRN 4359422}  (\bibinfo{year}{2023}).
\bibitem[{Li et~al.(2022)Li, Liang, and Pang}]{li2022comparison}
\bibinfo{author}{Y.~Li}, \bibinfo{author}{Z.~Liang}, \bibinfo{author}{S.~Pang},
\newblock \bibinfo{title}{Comparison between mean-variance and monotone
  mean-variance preferences under jump diffusion and stochastic factor model},
\newblock \bibinfo{journal}{arXiv preprint arXiv:2211.14473}
  (\bibinfo{year}{2022}).
\bibitem[{{\v{C}}ern{\`y}(2020)}]{cerny2020semimartingale}
\bibinfo{author}{A.~{\v{C}}ern{\`y}},
\newblock \bibinfo{title}{Semimartingale theory of monotone mean-variance
  portfolio allocation},
\newblock \bibinfo{journal}{Mathematical Finance} \bibinfo{volume}{30}
  (\bibinfo{year}{2020}) \bibinfo{pages}{1168--1178}.
\bibitem[{Li et~al.(2024)Li, Liang, and Pang}]{li2024MMVPS}
\bibinfo{author}{Y.~Li}, \bibinfo{author}{Z.~Liang}, \bibinfo{author}{S.~Pang},
\newblock \bibinfo{title}{Continuous-time monotone mean-variance portfolio
  selection},
\newblock \bibinfo{journal}{arXiv preprint arXiv:2211.12168}
  (\bibinfo{year}{2024}).
\bibitem[{Li and Guo(2021)}]{LG2021RAIRO-OR}
\bibinfo{author}{B.~Li}, \bibinfo{author}{J.~Guo},
\newblock \bibinfo{title}{Optimal reinsurance and investment strategies for an
  insurer under monotone mean-variance criterion},
\newblock \bibinfo{journal}{RAIRO-Operations Research} \bibinfo{volume}{55}
  (\bibinfo{year}{2021}) \bibinfo{pages}{2469--2489}.
\bibitem[{Li et~al.(2023)Li, Guo, and Tian}]{li2023optimal}
\bibinfo{author}{B.~Li}, \bibinfo{author}{J.~Guo}, \bibinfo{author}{L.~Tian},
\newblock \bibinfo{title}{Optimal investment and reinsurance policies for the
  {C}ram{\'e}r-{L}undberg risk model under monotone mean-variance preference},
\newblock \bibinfo{journal}{International Journal of Control}
  (\bibinfo{year}{2023}) \bibinfo{pages}{1--15}.
\bibitem[{Shen and Zou(2022)}]{shen2022cone}
\bibinfo{author}{Y.~Shen}, \bibinfo{author}{B.~Zou},
\newblock \bibinfo{title}{Cone-constrained monotone mean-variance portfolio
  selection under diffusion models},
\newblock \bibinfo{journal}{SIAM Journal on Financial Mathematics}
  \bibinfo{volume}{13} (\bibinfo{year}{2022}) \bibinfo{pages}{SC99--SC112}.
\bibitem[{Hu et~al.(2023)Hu, Shi, and Xu}]{hu2023constrained}
\bibinfo{author}{Y.~Hu}, \bibinfo{author}{X.~Shi}, \bibinfo{author}{Z.~Q. Xu},
\newblock \bibinfo{title}{Constrained monotone mean-variance problem with
  random coefficients},
\newblock \bibinfo{journal}{SIAM Journal on Financial Mathematics}
  \bibinfo{volume}{14} (\bibinfo{year}{2023}) \bibinfo{pages}{838--854}.
\bibitem[{Xia and Yan(2006)}]{XYincomplete2006}
\bibinfo{author}{J.~Xia}, \bibinfo{author}{J.-A. Yan},
\newblock \bibinfo{title}{Markowitz's portfolio optimization in an incomplete
  market},
\newblock \bibinfo{journal}{Mathematical Finance} \bibinfo{volume}{16}
  (\bibinfo{year}{2006}) \bibinfo{pages}{203--216}.
\bibitem[{Yong and Zhou(1999)}]{ZXY1999stochastic}
\bibinfo{author}{J.~Yong}, \bibinfo{author}{X.~Y. Zhou},
  \bibinfo{title}{Stochastic Controls: Hamiltonian Systems and HJB Equations},
  volume~\bibinfo{volume}{43}, \bibinfo{publisher}{Springer Science \& Business
  Media}, \bibinfo{year}{1999}.
\bibitem[{Biagini and {\v{C}}ern{\`y}(2011)}]{biagini2011admissible}
\bibinfo{author}{S.~Biagini}, \bibinfo{author}{A.~{\v{C}}ern{\`y}},
\newblock \bibinfo{title}{Admissible strategies in semimartingale portfolio
  selection},
\newblock \bibinfo{journal}{SIAM Journal on Control and Optimization}
  \bibinfo{volume}{49} (\bibinfo{year}{2011}) \bibinfo{pages}{42--72}.

\end{thebibliography}
	
\end{document}